\documentclass[12pt]{amsart}
\usepackage{amssymb, amsmath, amsthm, cancel, hyperref, color, enumerate, verbatim, mathtools}
\usepackage{multirow} % so that Table 2 works
\usepackage{stmaryrd} % for llbracket and llparenthesis

\usepackage[margin=1in, includeheadfoot]{geometry}
\newtheorem{thm}{Theorem}
\newtheorem{rem}{Remark}
\newtheorem{lem}{Lemma}

\newtheorem{prop}{Proposition}

\newcommand{\ZZ}{\mathbb{Z}}
\newcommand{\NN}{\mathbb{N}}
\newcommand{\QQ}{\mathbb{Q}}
\newcommand{\CC}{\mathbb{C}}

\newtheoremstyle{case}{}{}{}{}{}{:}{ }{}
\theoremstyle{case}

\newcommand{\HH}{\mathbb{H}}

\newcommand{\FF}{\mathbb{F}}

\newcommand{\mc}[1]{\mathcal{#1}}

\newcommand{\SL}{\text{SL}}
\newcommand{\ord}{\text{ord}}

\usepackage{stmaryrd} % for llbracket and llparenthesis
\usepackage[normalem]{ulem}

\usepackage{xcolor}
\title{Ramanujan congruences for overpartitions with restricted odd differences}
\author{Michael Hanson \and Jeremiah Smith}

\begin{document}

\maketitle

\begin{abstract}
We investigate Ramanujan congruences for the function $\overline{t}(n)$, which counts the overpartitions of $n$ with restricted odd differences. In particular, we show that only one such congruence exists. Our method involves using the theory of modular forms to prove a more general theorem which bounds the number of primes possible for Ramanujan congruences in certain eta-quotients. This generalizes work done by Jonah Sinick. We also provide two congruences modulo 5 for $\overline{t}(n)$. 
\end{abstract}

\section{Introduction}

Perhaps one of the most famous results of Ramanujan was his collection of congruences for the partition function $p(n)$. For a positive integer $n$, $p(n)$ denotes the number of partitions of $n$, i.e. the number of ways to write $n$ as a non-increasing sum of positive integers. For example, there are five partitions of the integer 4, those being $4, 3 + 1$, $2 + 2$, $2 + 1 + 1$, $1 + 1 + 1 + 1$, and so $p(4) = 5$. Here we define $p(0) := 1$ and $p(n) = 0$ when $n < 0$. It is common notation to write a partition $n = \lambda_{1} + \lambda_{2} + \cdots + \lambda_{r}$ of $n$ in the concatenated form $\lambda_{1}\lambda_{2}\ldots\lambda_{r}$. Ramanujan's famous result on $p(n)$ is the collection of congruences
        \begin{align}\label{Ramanujan}
        \begin{cases} p(5n+4) \equiv 0 \pmod{5}, \\
        p(7n+5) \equiv 0 \pmod{7}, \\
        p(11n+6) \equiv 0 \pmod{11}.
        \end{cases}
        \end{align}
These are proven starting with the fact that the generating function for $p(n)$ takes the form $\sum_{n \geq 0}p(n)q^{n} = \prod_{n\geq 1} (1-q^{n})^{-1}$ with $q := e^{2\pi i z}$, which is essentially the inverted Dedekind's eta-function $\eta(z) := q^{1/24}\prod_{n\geq 1}(1 - q^{n})$. Ramanujan's congruences for $p(n)$ all take the form $p(\ell n + a) \equiv 0 \pmod{\ell}$ for some prime $\ell$, hence all such congruences are called \textit{Ramanujan congruences}. It is natural to ask whether there are other primes $\ell$ for which $p(n)$ has Ramanujan congruences. But as it turns out, work of Ahlgren and Boylan \cite{Ahlgren-Boylan} show that the Ramanujan congruences \eqref{Ramanujan} are the only ones for $p(n)$.

Ramanujan's work has inspired many others to consider such congruences for other modified partition functions. In this work, we consider the modified partition function $\overline{t}(n)$ that counts  ``overpartitions with restricted odd parts," as originally described in \cite{OGt}. An \textit{overpartition} of $n$ is a partition of $n$ in which the final occurrence of a number may be overlined. We let $\overline{p}(n)$ be the number of overpartitions of $n$. Thus, for example, $\overline{p}(4) = 14$, with the 14 partitions given as
    \begin{gather*}
    4, \quad \overline{4}, \quad 3 + 1, \quad \overline{3} + 1, \quad 3 + \overline{1}, \quad \overline{3} + \overline{1}, \quad 2 + 2, \quad 2 + \overline{2}, \quad 2 + 1 + 1, \\
    \overline{2} + 1 + 1, \quad 2 + 1 + \overline{1}, \quad \overline{2} + 1 + \overline{1}, \quad 1 + 1 + 1 + 1, \quad 1 + 1 + 1 + \overline{1}.
    \end{gather*}
Our function of interest $\overline{t}(n)$ counts the number of overpartitions of $n$ with the following restrictions.
\begin{enumerate}[(i)]
\item The difference between two successive parts may be odd only if the larger part is overlined.
\item If the smallest part is odd, then it is overlined. 
\end{enumerate}
For example, using $n = 4$ again we have $\overline{t}(4) = 8$, with the 8 such partitions given as
    \begin{gather*}
    4, \quad \overline{4}, \quad 3 + \overline{1}, \quad \overline{3} + \overline{1}, \quad 2 + 2, \quad 2 + \overline{2}, \quad \overline{2} + 1 + \overline{1}, \quad 1 + 1 + 1 + \overline{1}.
    \end{gather*}

The authors of \cite{OGt} show that the generating function for $\overline{t}(n)$ is given by 
        $$\sum_{n\geq 0} \overline{t}(n)q^{n} = \frac{\eta(3z)}{\eta(2z)\eta(z)}.$$
This is a weakly holomorphic modular form of weight $-1/2$ for the congruence subgroup $\Gamma_{1}(144)$. Many congruences have been proven for $\overline{t}(n)$. In particular, congruences modulo 2, 3, and 5 have been found:
    \begin{itemize}
    \item Theorem 1.1 of \cite{H-S}: For $n \geq 0$ we have $$\overline{t}(n) \equiv \begin{cases} (-1)^{k+1} \pmod{3} & n = k^{2} \text{ some $k$,} \\ 0 \pmod{3} & \text{else.} \end{cases}$$ 
    \item Theorem 1.2 of \cite{H-S}: For $n \geq 1$ we have $$\overline{t}(2n) \equiv \begin{cases} 1 \pmod{2} & n = (3k+1)^{2} \text{ some $k$,} \\ 0 \pmod{2} & \text{else.} \end{cases}$$
    \item Theorem 1.1 of \cite{LinLiu}: For all $\alpha, n \geq 0$ we have $$\overline{t}(9^{\alpha}(45n+30)) \equiv 0 \pmod{5}.$$
    \end{itemize}
More congruences can be found in \cite{Chern-Hao, H-S, LinLiu, Naika}. We provide two more congruences modulo 5, proven in Section \ref{mod5} using the theory of modular forms. 

\begin{thm}\label{thm mod 5}
The following congruences hold for all $n \geq 0$:
\begin{align}
    \bar{t}(80n + 40) &\equiv 0 \pmod{5}, \label{congruence 1}\\ 
    \bar{t}(80n + 60) &\equiv 0 \pmod{5} \label{congruence 2}.
\end{align}
\end{thm}

However, one might question if any Ramanujan congruences hold for $\overline{t}(n)$ (i.e. congruences of the form $\overline{t}(\ell n + a) \equiv 0 \pmod{\ell}$ for a prime $\ell$). The results of \cite{H-S} above show that the only Ramanujan congruence mod 2 or 3 is $\overline{t}(3n+2) \equiv 0 \pmod{3}$, and one can check that there are no Ramanujan congruences mod 5. In fact, we have the following.

\begin{thm}\label{2nd place}
The only Ramanujan congruence for $\overline{t}(n)$ is \begin{equation*}
    \overline{t}(3n+2) \equiv 0 \pmod{3}.
\end{equation*}
\end{thm}

In order to prove Theorem \ref{2nd place}, it suffices to show that there are no Ramanujan congruences for primes $\ell>5.$ This is an immediate consequence of Theorem \ref{mainthm} below, which is a generalization of work done by Sinick in \cite{Sinick}. 

\begin{thm}\label{mainthm}
Let $\lambda = \lambda_{1}\lambda_{2} \ldots \lambda_{r}$ and $\mu = \mu_{1}\mu_{2} \ldots \mu_{s}$ be partitions of $u \in \NN$ and $v \in \NN$, respectively, where we assume without loss of generality that $\lambda_i\neq \mu_j$ for all $i,j.$ Define
        $$f(z) := \prod_{n \geq 1} \frac{(1 - q^{\lambda_{1}n})(1 - q^{\lambda_{2}n})\cdots (1 - q^{\lambda_{r}n})}{(1 - q^{\mu_{1}n})(1 - q^{\mu_{2}n})\cdots (1 - q^{\mu_{s}n})} =: \sum_{n\geq 0} c(n)q^{n}.$$ 
Let $N := \text{lcm}(\lambda_{1}, \ldots, \lambda_{r}, \mu_{1}, \ldots, \mu_{s})$, and let $\gamma$ be the number of occurrences of the smallest element of $\{\lambda_{1}, \ldots, \lambda_{r}, \mu_{1}, \ldots, \mu_{s}\}.$  Let $\ell$ be prime such that $\ell > \max(5, |s-r|+4)$ and $(\ell, \gamma N)=1$. If either $s-r \in 2\NN_{0}$ or $u \equiv v \pmod{\ell}$ and $r-s \neq 1,3$, then $c(n)$ does not obey a Ramanujan congruence modulo $\ell$.
\end{thm}

This theorem identifies an explicit upper bound for the primes that could result in a Ramanujan congruence. 

In Section \ref{forms mod l}, we provide a brief introduction to modular forms modulo $\ell$, and we give a few results within the theory that will be used in later sections. Section \ref{keys to theorem} provides motivation and an outline for the proof of Theorem \ref{mainthm}. Section \ref{calcFiltrations} proves the most interesting and essential result for the proof of Theorem \ref{mainthm}. Section \ref{calculations} disposes of some necessary calculations needed for Section \ref{pf of main} which complete the proofs of Theorems \ref{mainthm} and \ref{2nd place}, respectively.
Section \ref{mod5} proves Theorem \ref{thm mod 5}. Section \ref{exampleofProp2} provides a helpful example pertaining to the discussion following Proposition \ref{filtration}. 

\section{Modular forms modulo $\ell$}\label{forms mod l}

In this section, we introduce the notion of ``modular forms mod $\ell$" as well as some preliminary results that will be used in later sections. We refer the reader to \cite{Web} for a more detailed account of this material. 

We denote the $\CC$-vector space of weakly holomorphic modular forms of integer weight $k$ on the congruence subgroup $\Gamma_{1}(N)$ of $\SL_{2}(\ZZ)$ by $M^{!}_{k}(\Gamma_{1}(N))$. Let $M_{k}(\Gamma_{1}(N))$ denote the subspace of those forms which are holomorphic at the cusps of $\Gamma_{1}(N)$, and let $S_{k}(\Gamma_{1}(N))$ denote its corresponding subspace of cusp forms. Given $M = \begin{psmallmatrix} a & b \\ c & d \end{psmallmatrix} \in \SL_{2}(\ZZ)$, define the usual weight-$k$ slash operator on holomorphic functions $f$ on the upper half plane $\HH$ as 
        $$f \mid_{k} M := (cz+d)^{-k}f\left(\frac{az+b}{cz+d}\right).$$
We will sometimes drop the subscript $k$ for notational convenience. Now define $\theta := \frac{1}{2\pi i} \frac{d}{dz} = q\frac{d}{dq}$, where $q := e^{2\pi i z}$, so that on Fourier series we have
        $$\theta \left(\sum_{n \geq 0} a(n)q^{n}\right) = \sum_{n\geq 0} na(n)q^{n}.$$
Also define the $m$-th $U$-operator on Fourier series as
        $$\sum_{n\geq 0} a(n)q^{n} \mid U_{m} := \sum_{n\geq 0} a(mn)q^{n}.$$

Eisenstein series are canonical examples of modular forms for $\SL_{2}(\ZZ)$, and they play an important role in Lemma \ref{R} below. For even $k > 2$, the weight-$k$ \textit{Eisenstein series} for $\SL_{2}(\ZZ)$ is
        $$E_{k}(z) := 1 - \frac{2k}{B_{k}}\sum_{n\geq 1} \sigma_{k-1}(n)q^{n},$$
where $B_{k}$ is the $k$-th Bernoulli number, and $\sigma_{k-1}(n) = \sum_{d \mid n} d^{k-1}$. When $k = 2$, $E_{2}(z) \coloneqq 1 - 24\sum_{n\geq 1}\sigma_1(n)q^n$ is not a modular form, but rather a ``quasi-modular form." It has the transformation law
        \begin{align}\label{E2}
        E_{2}(z) \mid M = E_{2}(z) - \frac{6ic}{\pi(cz+d)},
        \end{align}
for $M = \begin{psmallmatrix} a & b \\ c & d \end{psmallmatrix} \in \SL_{2}(\ZZ)$. We also have the \textit{Delta function}
        $$\Delta(z) = \eta(z)^{24} = q\prod_{n=1}^{\infty}(1-q^{n})^{24} \in M_{12}(\SL_{2}(\ZZ)),$$
which vanishes at the cusp $\infty$ of $\SL_{2}(\ZZ)$. 

Given a modular form $f \in M_{k}(\Gamma_{1}(N)) \cap \ZZ\llbracket q \rrbracket$, one can reduce the Fourier coefficients of $f$ modulo a prime $\ell$, giving an element $\tilde{f}$ of $\FF_{\ell}\llbracket q \rrbracket$. We call $\tilde{f}$ a \textit{modular form modulo $\ell$} for $\Gamma_{1}(N)$. The \textit{filtration} of $f$ is defined as
        $$w_{\ell}(f) := \min\{k' : \tilde{f} \in \widetilde{M}_{k'}(\Gamma_{1}(N))\},$$
where
        $$\widetilde{M}_{k'}(\Gamma_{1}(N)) := \{\tilde{g} : g \in M_{k'}(\Gamma_{1}(N))\}.$$
We will also refer to preimages of $\tilde{f}$ under the reduction map as ``modular forms modulo $\ell$". 

By using \eqref{E2}, one can easily generalize Lemma 3 of \cite{Swinnerton-Dyer} for $N > 1$ to conclude that if $f \in M_{k}(\Gamma_{1}(N))$ then $12\theta f - k E_{2} f \in M_{k+2}(\Gamma_{1}(N))$. Theorem 2(i) from \cite{Swinnerton-Dyer} implies that $E_{\ell - 1} \equiv 1 \pmod{\ell}$ and $E_{\ell + 1} \equiv E_{2} \pmod{\ell}$. These facts come together to prove Lemma \ref{R}.

\begin{lem}[Lemma 2.1 of \cite{Sinick}]\label{R}
If $f \in M_{k}(\Gamma_{1}(N)) \cap \ZZ\llbracket q \rrbracket$, then defining $R$ to be
        \begin{align}\label{R-eq}
        R := \left( \theta f - \frac{k}{12}E_{2}f\right)E_{\ell - 1} + \frac{k}{12}E_{\ell + 1}f,
        \end{align}
$R$ is a modular form of weight $k + \ell + 1$ such that $R \equiv \theta f \pmod{\ell}$. In particular, $\theta f$ is a modular form $\pmod{\ell}$ for $\Gamma_{1}(N)$. It follows that if $\tilde{f} \not\equiv 0 \pmod{\ell}$, then $w_{\ell}(\theta f) \leq w_{\ell}(f) + \ell + 1$. 
\end{lem}

We will also need the following facts about filtrations.

\begin{lem}\label{2.2}
Let $N \geq 4$, let $f, g \in M(\Gamma_{1}(N)) \cap \ZZ\llbracket q \rrbracket$, and let $\ell \geq 5$ be prime. Then:
\begin{enumerate}[(i)]

\item We have $w_{\ell}(\theta f) = w_{\ell}(f) + \ell + 1$ if and only if $w_{\ell}(f) \not\equiv 0 \pmod{\ell}$. 

\item If $f$ and $g$ have weights $k_{1}$ and $k_{2}$ respectively and $\tilde{f} \equiv \tilde{g} \not\equiv 0 \pmod{\ell}$, then $k_{1} \equiv k_{2} \pmod{\ell - 1}$. 

\item If $\ell \nmid N$ then for $i \geq 0$ we have $w_{\ell}(f^{i}) = i\cdot w_{\ell}(f)$.  
\end{enumerate}
\end{lem}

The proofs of (ii) and the reverse implication of (i) are given directly in Section 4 of \cite{Gross}. The remaining facts are quick consequences of the results given in the same section.

The following elementary fact will be useful in Section \ref{pf of main}: if $f \in M_{k}(\Gamma_{1}(N)) \cap \ZZ\llbracket q \rrbracket$ and $\ell$ is a prime, then $(f \mid U_{\ell})^{\ell} \equiv f - \theta^{\ell - 1}f \pmod{\ell}$. It follows that
        $$f \mid U_{\ell} \equiv 0 \pmod{\ell} \iff \theta^{\ell - 1}f \equiv f \pmod{\ell}.$$

\section{Keys to the Proof of Theorem \ref{mainthm}} \label{keys to theorem}
The statement of Theorem \ref{mainthm} concerns the existence of Ramanujan congruences for a particular type of eta-quotient. We hope to apply the following proposition originally due to I.~Kiming and J.~Olsson \cite{KimingandOlsson}  and then corrected by J.~Sinick \cite[Proposition 3.2]{Sinick}.

\begin{prop}\label{3options}
Let $\ell \geq 5$ be prime and $N \geq 4$, $\ell \nmid N$. Suppose that $f(z) \in M_{k}(\Gamma_{1}(N))$ has $\ell$-integral Fourier coefficients, $w_{\ell}(f(z)) \not\equiv 0 \pmod{\ell}$, and $\theta(f(z)) \not\equiv 0 \pmod{\ell}$. Suppose further that $w_{\ell}(\theta^{m}f(z)) \geq w_{\ell}(f(z))$. Then if the Fourier coefficients $d(n)$ of $f(z)$ satisfy $d(\ell n + b) \equiv 0 \pmod{\ell}$, one of the following is true: $b = 0$, $w_{\ell}(f(z)) \equiv (\ell + 1)/2 \pmod{\ell}$, or $w_{\ell}(f(z)) \equiv (\ell + 3)/2 \pmod{\ell}.$
\end{prop}

There is a glaring problem with naively applying this proposition to the eta-quotient in Theorem $\ref{mainthm}$: the fact that the eta-quotient is not necessarily an integer weight holomorphic modular form. We will fix this by defining an integer weight modular form for each prime $\ell$ for which a Ramanujan congruence exists mod $\ell$ if and only if a Ramanujan congruence exists mod $\ell$ for the given eta-quotient. This definition is given below. 

Let $N := \text{lcm}(\lambda_{1}, \ldots, \lambda_{r}, \mu_{1}, \ldots, \mu_{s})$. Define
		$$F_{\ell}(z) := \Delta(z)^{\ell^{t}}\left( \frac{\Delta(\mu_{1}z)\Delta(\mu_{2}z)\cdots \Delta(\mu_{s}z)}{\Delta(\lambda_{1}z)\Delta(\lambda_{2}z)\cdots \Delta(\lambda_{r}z)} \right)^{\delta_{\ell}} =: \sum_{n \geq 0} D(n)q^{n},$$
where $\delta_{\ell} := \frac{\ell^{2} - 1}{24}$, and where $t \geq 2$ is the smallest integer such that $F_{\ell}$ is holomorphic at the cusps of $\Gamma_{1}(N)$.

\begin{rem}
Note that $F_{\ell}$ is a modular form of weight $\frac{(\ell^{2} - 1)(s-r)}{2} + 12\ell^{t}$ for $\Gamma_{1}(N)$. In what follows, we sometimes need $N \geq 4$. We can substitute $4N$ for $N$ without loss of generality when $N < 4$. 
\end{rem} 
The fact that the Ramanujan congruences for $F_{\ell}$ are in correspondence with those of $f$ is the following lemma.

\begin{lem}\label{cong-reduce}
With notation as above, we have that $D(\ell n + b) \equiv 0 \pmod{\ell}$ if and only if $c(\ell n + a) \equiv 0 \pmod{\ell}$, where $b$ is defined by $24a \equiv 24b + (u - v) \pmod{\ell}$. 
\end{lem}

The proof of this will be given later in Section \ref{calculations}. Notice that if $u \equiv v \mod{\ell}$, then $a \equiv b \pmod{\ell}$ (recall $\ell > 5$). Now that we have defined an appropriate modular form, we need to proceed by checking that $F_\ell$ satisfies the other assumptions of Proposition \ref{3options}. The fact that $\theta F_\ell \not\equiv 0 \pmod{\ell}$ is a simple calculation (see Proposition \ref{non-vanishing}) which will be done in Section \ref{calculations}. On a technical note, the necessity of $\gamma$ in the statement of Theorem $\ref{mainthm}$ comes from this calculation. The most difficult assumptions to verify in Proposition $\ref{3options}$ are precisely those which deal with the filtrations. This will be accomplished in Section \ref{calcFiltrations}. By these calculations, we will see that the two congruences $w_\ell(F_\ell(z)) \equiv (\ell+1)/2 \pmod{\ell}$ and $w_\ell(F_\ell(z)) \equiv (\ell + 3)/2 \pmod{\ell}$ are impossible. Lastly, in order to get the full strength of Theorem \ref{mainthm}, we must dispose of the possibility that $b = 0$ in the statement of Proposition \ref{3options}. This is a technical point which is resolved in Section \ref{pf of main}. 

\section{Calculating the Filtrations} \label{calcFiltrations}
In this section, we show that $w_{\ell}(\theta^{m} F_{\ell}) \geq w_{\ell}(F_{\ell}) = \frac{(\ell^{2} - 1)(s-r)}{2} + 12\ell^{t}$ by proving the more general Proposition \ref{filtration}.

\begin{prop}\label{filtration}
Let $F \in M_k(\Gamma_1(N)) \cap \mathbb{Z}[[q]]$, $\ell \geq 5$ prime, $\ell \nmid N$, $\theta F \not\equiv 0 \pmod{\ell}$, and suppose that $F$ does not vanish on $\HH$. Then $w_\ell(F) = k$ and $w_\ell(\theta^m F) \geq w_\ell(F)$.
\end{prop}

Following the discussion below Lemma 4.1 of \cite{Sinick}, we enumerate the cosets of $\Gamma_{1}(N)$ in $\SL_{2}(\ZZ)$ by $\{i\}_{1 \leq i \leq 2d_{N}}$. Let $M_{i}$ be a representative of the $i$-th coset. Let $\alpha_{i}$ be the cusp that $M_{i}$ sends to $\infty$. Denote the minimal period of $F \mid M_{i}$ by $t_{i}$. Then $F \mid M_{i}$ has a Fourier expansion in powers of $q_{t_{i}} := e^{2\pi i z/t_{i}}$, and the order of vanishing of $F$ at $\alpha_{i}$ is the index of the first non-vanishing Fourier coefficient of $F$ in powers of $q_{t_{i}}$, denoted $\ord_{\alpha_{i}}(F)$. Though these $q_{t_{i}}$-Fourier expansions of $F$ need not have coefficients in $\ZZ$, Corollary 5.3 of \cite{Radu} tells us that they lie in $\ZZ[\zeta_{N}]$ with $\zeta_{N}$ a primitive $N$-th root of unity. 

Instead of considering modular forms modulo $\ell$, one may choose an algebraic number field $L$ and look at forms $g \in M_{k}(\Gamma_{1}(N)) \cap L\llbracket q \rrbracket$. We can then reduce $g$ modulo $v$ for any prime $v \in \mc{O}_{L}$ such that the $v$-adic valuation of $g$ is $0$. This allows us to define the notion of ``modular form modulo $v$", and we can define the filtration $w_{v}$ for nonvanishing forms $\pmod{v}$ in the obvious way. Hence we can define the $v$-adic valuation of the corresponding power series to be the minimum of the $v$-adic valuations of the coefficients. Defining $\widetilde{\ord}_{\alpha_{i}}(f)$ to be the order of vanishing of $f \pmod{v}$ at the cusp $\alpha_{i}$ (this is well-defined; see for example Remark 2.4 of \cite{Dewar}), we have Lemma 4.2 of \cite{Sinick}, stated below.

\begin{lem}\label{tilde ord bound}
Let $m \geq 1$ be an integer and let $v$ be a prime in $\ZZ[\zeta_{N}]$ such that $v \nmid 2, 3, N$. Let $f(z)$ be a modular form for $\Gamma_{1}(N)$ such that $f(z) \mid M_{i}$ has coefficients in $\QQ(\zeta_{N})$ and $v$-adic valuation $0$. Let $\alpha_{i}$ be a cusp of $\Gamma_{1}(N)$. Then
        $$\widetilde{\ord}_{\alpha_{i}}(\theta^{m}f) \geq \widetilde{\ord}_{\alpha_{i}}(f).$$
\end{lem}

\begin{rem}
Sinick's proof states, ``Since $v \nmid N$ and $f(z) \mid M_i$ has $v$-adic valuation 0, the Fourier expansion of $\theta(f\mid M_i)$ has $v$-adic valuation 0." However, this statement is false by taking a prime above $5$ in $\mathbb{Z}[\zeta_7]$ and considering $E_4(z)$ as a modular form for $\Gamma_1(7)$ for example. One needs the additional assumption that $\theta f \not\equiv 0 \pmod{v}$.
\end{rem}
\begin{proof}[Proof of Proposition \ref{filtration}]
We let $M_{i}$, $\alpha_{i}$, $t_{i}$, and $v \in \ZZ[\zeta_{N}]$ be as above, $v$ being a prime above $\ell$. Since $\theta F \not\equiv 0 \pmod{\ell}$ by assumption, $F \not\equiv 0 \pmod{\ell}$. In particular, $F \not\equiv 0 \pmod{v}$, and so Theorem 12.3.4 and Remark 12.3.5 of \cite{Diamond-Im} assert that $F \mid M_{i} \not\equiv 0 \pmod{v}$. Define
        $$G(z) := \prod_{i=1}^{2d_{N}} (F \mid M_{i}),$$
a modular form for $\SL_{2}(\ZZ)$ of weight $2d_{N}k$. As $F$ is zero-free on $\HH$, so is $G$, and so by the valence formula we have that $G$ is a non-zero constant multiple of $\Delta(z)^{e}$ with $e := \frac{2d_{N}k}{12} = \frac{d_{N}k}{6}$. It follows that $w_v(G)=12e$ since if there existed a modular form of smaller weight which is congruent to $G,$ Sturm's Theorem \cite[Theorem 2.58]{Web} would imply that it would have to vanish mod $v.$ Thus $w_v(F)=k,$ and since $F \in \ZZ\llbracket q \rrbracket$, we in fact have that $w_{\ell}(F) = k$, as desired. 

Now we show that $w_{\ell}(\theta^{m}F) \geq w_{\ell}(F)$. Notice that $\widetilde{\ord}_{\infty}(G) = e$, whence $\widetilde{\ord}_{\infty}(F \mid M_{i}) = \frac{\widetilde{\ord}_{\alpha_{i}}(F)}{t_{i}}$. Thus
        $$\sum_{i=1}^{2d_{N}} \frac{\widetilde{\ord}_{\alpha_{i}}(F)}{t_{i}} = \widetilde{\ord}_{\infty}(G) = e.$$
Define 
        $$H := \prod_{i=1}^{2d_{N}} (\theta^{m} F) \mid M_{i}.$$
Notice that $H$ is a modular form modulo $v$ for $\SL_{2}(\ZZ)$ by the modified version of Lemma \ref{R}, where we replace $\ell$ with $v$. Then
        $$\widetilde{\ord}_{\infty}(H) = \sum_{i=1}^{2d_{N}} \frac{\widetilde{\ord}_{\alpha_{i}}(\theta^{m}F)}{t_{i}} \geq \sum_{i=1}^{2d_{N}} \frac{\widetilde{\ord}_{\alpha_{i}}(F)}{t_{i}} = e,$$
the inequality being a consequence of Lemma \ref{tilde ord bound}. Since $\theta^{m}F$ is a modular form modulo $v$ which does not vanish, Theorem 12.3.4 and Remark 12.3.5 of \cite{Diamond-Im} assert that $(\theta^{m}F) \mid M_{i} \not\equiv 0 \pmod{v}$. This shows that $H \not\equiv 0 \pmod{v}$. Sturm's Theorem \cite[Theorem 2.58]{Web} now tells us that $w_{v}(H) \geq 12e$, and so $w_{v}(\theta^{m}F) = w_{\ell}(\theta^{m}F) \geq k = w_{\ell}(F)$, the first equality coming from the fact that $F$ has integral Fourier coefficients. 
\end{proof}

\section{Necessary Calculations} \label{calculations}
In this section, we produce two calculations that are necessary for the application of Proposition $\ref{3options}$.

\begin{prop} \label{non-vanishing}
We have that $\theta F_{\ell} \not\equiv 0 \pmod{\ell}$. 
\end{prop}

\begin{proof}
We note that by definition of $\theta$ it is sufficient to compute the Fourier expansion modulo $\ell$. We have
\begin{align*}
    F_{\ell}(z) &= \Delta(z)^{\ell^{t}}\left( \frac{\Delta(\mu_{1}z)\Delta(\mu_{2}z)\cdots \Delta(\mu_{s}z)}{\Delta(\lambda_{1}z)\Delta(\lambda_{2}z)\cdots \Delta(\lambda_{r}z)} \right)^{\delta_{\ell}} \\
    &= q^{\ell^t + (v-u)\delta_\ell} \prod_{n=1}^\infty \left[(1-q^n)^{\ell^t} \left(\prod_{i} (1-q^{\mu_i n})^{\ell^2-1}\right) \left(\prod_{j} \frac{1}{(1-q^{\lambda_j n})^{\ell^2-1}}\right)\right].
\end{align*}
Using a geometric series expansion, we have
\begin{align*}
    F_\ell(z) &= q^{\ell^t + (v-u)\delta_\ell} \prod_{n=1}^\infty \left[ (1-q^n)^{\ell^t} \left(\prod_{i} (1-q^{\mu_i n})^{\ell^2-1}\right) \left(\prod_{j} \left(\sum_{k=0}^\infty q^{\lambda_j n k}\right)^{\ell^2-1}\right)\right] \\
    &=q^{\ell^t + (v-u)\delta_\ell} \prod_{n=1}^\infty \left[(1-q^n)^{\ell^t} \left(\prod_{i} (1-q^{\mu_i n})^{\ell^2-1}\right) \left(\prod_{j} (1+q^{\lambda_j n} + \ldots)^{\ell^2-1}\right) \right]\\
    &= q^{\ell^t + (v-u)\delta_\ell} \prod_{n=1}^\infty\left[ (1-q^n)^{\ell^t} \left(\prod_{i} (1-(\ell^2-1)q^{\mu_i n} + \ldots ) \right) \left(\prod_{j} (1+(\ell^2-1)q^{\lambda_j n} + \ldots)\right)\right] \\
    &\equiv q^{\ell^t + (v-u)\delta_\ell} \prod_{n=1}^\infty \left[(1-q^n)^{\ell^t} \left(\prod_{i} (1+q^{\mu_i n} + \ldots )\right) \left(\prod_{j} (1-q^{\lambda_j n} + \ldots)\right) \right]\pmod{\ell} \\
    &\equiv q^{\ell^t + (v-u)\delta_\ell} \prod_{n=1}^\infty \left[(1-q^{\ell^t n})(1+ \alpha q^{\mu_r n} + \ldots)(1-\beta q^{\lambda_s n} + \ldots)\right] \pmod{\ell},
\end{align*}
where $\alpha$ is the number of occurrences of $\mu_r$ in $\mu$ and $\beta$ is the number of occurrences of $\lambda_s$ in $\lambda$. Hence,
\begin{align*}
    F_\ell(z) &\equiv q^{\ell^t + (v-u)\delta_\ell} (1\pm \cdots \pm \gamma q^{m} \pm \cdots)  \\
    &\equiv q^{\ell^t + (v-u)\delta_\ell} \pm \cdots \pm \gamma q^{m + \ell^t + (v-u)\delta_\ell} \pm \cdots \pmod{\ell},
\end{align*}
where $m = \min(\mu_r,\lambda_s)$. So, by applying the theta operator, 
\begin{align*}
    \theta F_\ell(z) &\equiv (\ell^t + (v-u)\delta_\ell)q^{\ell^t + (v-u)\delta_\ell} \pm \cdots \pm \gamma (m + \ell^t + (v-u)\delta_\ell)q^{m + \ell^t + (v-u)\delta_\ell} \pm \cdots \pmod{\ell}.
\end{align*}
One of these coefficients does not vanish mod $\ell$ because either $(v-u)\delta_\ell$ is divisible by $\ell$ or not, and from the fact that $(\gamma m,\ell) =1$ the result follows. 
\end{proof}

We also provide the proof of the  correspondence of Ramanujan congruences of the eta-quotient to an integer weight modular form.

\begin{proof}[Proof of Lemma \ref{cong-reduce}]
We have
		\begin{align*}
		\sum_{n \geq 0} c(n)q^{n} &= \prod_{n\geq 1} \frac{(1 - q^{\lambda_{1}n})\cdots (1 - q^{\lambda_{r}n})}{(1 - q^{\mu_{1}n})\cdots(1 - q^{\mu_{s}n})} \\
		&= \prod_{n\geq 1} \left( \frac{(1 - q^{\mu_{1}n})\cdots (1 - q^{\mu_{s}n})}{(1 - q^{\lambda_{1}n})\cdots(1 - q^{\lambda_{r}n})} \right)^{-\ell^{2}} \left( \frac{(1 - q^{\mu_{1}n})\cdots (1 - q^{\mu_{s}n})}{(1 - q^{\lambda_{1}n})\cdots(1 - q^{\lambda_{r}n})} \right)^{\ell^{2} - 1} \\
		&= \prod_{n\geq1} \left[ \left( \frac{(1 - q^{\mu_{1}n})\cdots (1 - q^{\mu_{s}n})}{(1 - q^{\lambda_{1}n})\cdots(1 - q^{\lambda_{r}n})} \right)^{-\ell^{2}} \frac{q^{\delta_{\ell}(\lambda_{1} + \cdots + \lambda_{r})}}{q^{\delta_{\ell}(\mu_{1} + \cdots + \mu_{s})}} \left( \frac{q^{\frac{1}{24}(\mu_{1}+ \cdots +\mu_{s})}(1-q^{\mu_{1}n})\cdots (1-q^{\mu_{s}n})}{q^{\frac{1}{24}(\lambda_{1}+\cdots +\lambda_{r})}(1-q^{\lambda_{1}n})\cdots (1-q^{\lambda_{r}n})} \right)^{\ell^{2} - 1} \right] \\
		&= \left( \prod_{n \geq 1} \frac{(1 - q^{\mu_{1}n})\cdots(1 - q^{\mu_{s}n})}{(1 - q^{\lambda_{1}n})\cdots(1 - q^{\lambda_{r}n})} \right)^{-\ell^{2}} q^{\delta_{\ell}(u-v)} \left(\frac{\Delta(\mu_{1}z)\cdots \Delta(\mu_{s}z)}{\Delta(\lambda_{1}z)\cdots \Delta(\lambda_{r}z)}\right)^{\delta_{\ell}}. 
		\end{align*}
Hence, 
		\begin{align*}
		q^{\delta_{\ell}(u - v)} \left( \frac{\Delta(\mu_{1}z)\cdots \Delta(\mu_{s}z)}{\Delta(\lambda_{1}z)\cdots \Delta(\lambda_{r}z)}\right)^{\delta_{\ell}} = \left( \prod_{n\geq 1} \frac{(1 - q^{\mu_{1}n})\cdots(1 - q^{\mu_{s}n})}{(1 - q^{\lambda_{1}n})\cdots (1 - q^{\lambda_{r}n})} \right)^{\ell^{2}}\sum_{n\geq 0} c(n)q^{n}.
		\end{align*}
Multiply both sides by $\Delta(z)^{\ell^{t}}$ to get
		\begin{align}\label{main}
		q^{\delta_{\ell}(u - v)} F_{\ell}(z) = \Delta(z)^{\ell^{t}}\left( \prod_{n\geq 1} \frac{(1 - q^{\mu_{1}n})\cdots(1 - q^{\mu_{s}n})}{(1 - q^{\lambda_{1}n})\cdots (1 - q^{\lambda_{r}n})} \right)^{\ell^{2}}\sum_{n\geq 0} c(n)q^{n}.
		\end{align}
Now we apply $U_{\ell}$ to equation \eqref{main}, reduce modulo $\ell$, and multiply both sides by $q^{-a}$ to get
		\begin{align*}
		\sum_{n \geq 0} D\left(\ell n + \delta_{\ell}( u - v)+ a\right)q^{n} &\equiv q^{-a}\Delta(z)^{\ell^{t-1}}\left( \prod_{n\geq 1} \frac{(1 - q^{\mu_{1}n})\cdots(1 - q^{\mu_{s}n})}{(1 - q^{\lambda_{1}n})\cdots (1 - q^{\lambda_{r}n})} \right)^{\ell}\sum_{n\geq 0} c(\ell n)q^{n} \pmod{\ell}.
		\end{align*}
Applying Proposition (3) of \cite{Ono} to the right-hand side, we get 
		$$\sum_{n\geq 0} D\left(\ell n + \delta_{\ell}( u - v)+ a\right)q^{n} \equiv 0 \pmod{\ell} \iff \sum_{n\geq 0} c(\ell(n - \ell^{t-1}) + a)q^{n} \equiv 0 \pmod{\ell}.$$
Since $c(n) = 0$ for $n < 0$ this proves the lemma.
\end{proof}

%%%%%%%%%%%%%%%%%%%%%%%%%%%%%%%%%%%%%%%%%%%%%%%%%%%%%%%%%%%%%%%%%%%%%%%%%%%%%%%%%%%%%%%%%%%%%%%%%%%%%%%%%%%%%%%%%%%%%%%%%%%%%%%%%%%%%%%%%%%%%%%%%%%%%%%%%%%%%%%%%%%%%%%%%%%%%%%%%%%%%%%%%%%%%%%%%%%%%%%%%%%%%%%%%%%%%%%%%%%%%%%%%%%%%%%%%%%%%%%%%%%%%%%%%%%%%%%%%%%%%%%%%%%%%%%%%%%%%%%%%%%%%%%%

%%%%%%%%%%%%%%%%%%%%%%%%%%%%%%%%%%%%%%%%%%%%%%%%%%%%%%%%%%%%%%%%%%%%%%%%%%%%%%%%%%%%%%%%%%%%%%%%
%%%%%%%%%%%%%%%%%%%%%%%%%%%%%%%%%%%%%%%%%%%%%%%%%%%%%%%%%%%%%%%%%%%%%%%%%%%%%%%%%%%%%%%%%%%%%%%%
%%%%%%%%%%%%%%%%%%%%%%%%%%%%%%%%%%%%%%%%%%%%%%%%%%%%%%%%%%%%%%%%%%%%%%%%%%%%%%%%%%%%%%%%%%%%%%%%

\section{Proof of Theorem \ref{mainthm}}\label{pf of main}
By using the results in Sections \ref{calcFiltrations} and \ref{calculations}, we may apply Proposition \ref{3options} to $F_\ell$. Notice that the two congruences $w_{\ell}(F_{\ell}(z)) \equiv (\ell + 1)/2 \pmod{\ell}$ and $w_{\ell}(F_{\ell}(z)) \equiv (\ell + 3)/2 \pmod{\ell}$ are impossible given the lower bound on $\ell$ and also the fact that $r-s \neq 1, 3$. This shows that if $F_\ell$ has a Ramanujan congruence, then $b = 0$. Thus, checking that $f$ has no Ramanujan congruences modulo $\ell$ is equivalent to checking that $\sum_{n \geq 0} D(\ell n)q^{n} \not\equiv 0 \pmod{\ell}$. In the case that $u \equiv v \pmod{\ell}$, we have that $F_\ell=q^{\ell M}(1+\cdots)$ for some $M$ and therefore $\sum_{n\geq 0} D(\ell n)q^{n} \not\equiv 0 \pmod{\ell}$.

The case $u \neq v$ and $s-r \in 2 \NN_{0}$ is more subtle, but luckily the proof of Theorem 1.2 in \cite{Sinick} works just as well in our situation. We start by stating Proposition 5.1 of \cite{Sinick}.

\begin{prop}[Proposition 5.1 of \cite{Sinick}]\label{cases}
Let $\ell \geq 5$ be prime and $N \geq 4$, $\ell \nmid N$. Suppose that $f(z) \in M_{k}(\Gamma_{1}(N))$ has $\ell$-integral Fourier coefficients, $w_{\ell}(f(z)) \not\equiv 0 \pmod{\ell}$, and $\theta f \not\equiv 0 \pmod{\ell}$. Suppose further that $w_{\ell}(\theta^{m} f(z)) \geq w_{\ell}(f(z))$. Let $i_{1} < i_{2} < \cdots < i_{c}$ be those $i \in \{0, 1, \ldots, \ell - 1\}$ for which $w_{\ell}(\theta^{i} f) \equiv 0 \pmod{\ell}$. Write $w_{\ell}(\theta^{i_{j} + 1} f) = w_{\ell}(\theta^{i_{j}} f) + (\ell + 1) - s_{j}(\ell - 1)$. Write $k = w_{\ell}(f)$ and let $k_{0} \in \{1, \ldots, \ell - 1\}$ be such that $k \equiv -k_{0} \pmod{\ell}$. Then one of the four cases below holds:
\begin{enumerate}[(I)]

\item $k \equiv 1 \pmod{\ell}$, $c = 1$, $i_{1} = \ell - 1$, and $s_{1} = \ell + 1$.

\item $k \equiv 2 \pmod{\ell}$, $c = 1$, $i_{1} = \ell - 2$, and $s_{1} = \ell + 1$. 

\item $k \not\equiv 1 \pmod{\ell}$, $c = 2$, $(i_{1}, i_{2}) = (k_{0}, \ell - 1)$, and $(s_{1}, s_{2}) = (k_{0} + 1, \ell - k_{0})$. 

\item $k \not\equiv 1 \pmod{\ell}$, $c = 2$, $(i_{1}, i_{2}) = (k_{0}, \ell - 2)$, and $(s_{1}, s_{2}) = (k_{0} + 2, \ell - k_{0} - 1)$. 

\end{enumerate}
We have $w_{\ell}(f) = w_{\ell}(\theta^{\ell - 1} f)$ if and only if case (II) or case (IV) holds.
\end{prop}

We proceed by contradiction by assuming $D(\ell n ) \equiv 0 \pmod{\ell}$ for all $n$ which implies that $\theta^{\ell - 1}F_{\ell} \equiv F_{\ell} \pmod{\ell}$. So by the last statement in Proposition \ref{cases}, we are in cases (II) or (IV). But in case (II) we have
        $$w_{\ell}(F_{\ell}) = \frac{(s-r)(\ell^{2} - 1)}{2} + 12\ell^{t} \equiv 2 \pmod{\ell} \iff r-s \equiv 4 \pmod{\ell},$$
which contradicts that $\ell > s-r + 4$. So we are in case (IV). Here, we have
        $$k_{0} \equiv -k \equiv \frac{s-r}{2} \pmod{\ell}.$$
Using Lemma \ref{2.2}, the identity $w_{\ell}(\theta^{i_{j} + 1} f) = w_{\ell}(\theta^{i_{j}} f) + (\ell + 1) - s_{j}(\ell - 1)$ with $s_{j} = s_{1} = k_{0} + 2$ becomes
        \begin{align}
        w_{\ell}(\theta^{k_{0}+1}F_{\ell}) &= w_{\ell}(F_{\ell}) + (\ell+1)(k_{0} + 1) - (k_{0} + 2)(\ell - 1) \nonumber\\
        &= w_{\ell}(F_{\ell}) + 2k_{0} + 3 - \ell. \label{filt}
        \end{align}
Now, $2k_{0} \equiv s-r \pmod{\ell}$. Since $s-r$ is even and $\ell > s-r \in \mathbb{N}_{0}$, we must have that
        $$k_{0} = \frac{s-r}{2}.$$
Thus, \eqref{filt} becomes
        $$w_{\ell}(F_{\ell}) + s-r + 3 - \ell.$$
But $\ell > s-r + 3$, so this implies that
        $$w_{\ell}(\theta^{k_{0} + 1}F_{\ell}) = w_{\ell}(F_{\ell}) + s-r + 3 - \ell < w_{\ell}(F_{\ell}),$$
contradicting Proposition \ref{filtration}. This concludes the proof of Theorem \ref{mainthm}. 

\section{Proof of Theorem \ref{thm mod 5}}\label{mod5}

In this section we use modular forms of half-integral weight. See \cite{Web} for background. We follow in the same theme as Ono \cite{Ono}. We first define the eta-quotient
\begin{equation*}
    f(z) \coloneqq  \frac{\eta(3z)}{\eta(2z)\eta(z)} \ \eta^{12}(80z) = \sum_{m \geq 0} b(m)q^{m}.
\end{equation*}
Using the appropriate theorems for eta-quotients (see for example \cite[Theorem 1.64]{Web}), one can see that $f \in M^!_{\frac{11}{2}}(\Gamma_1(1440))$. To make this holomorphic at the cusps, we introduce the factor $\eta^5(z)/\eta(5z)$. We make the following definition: \begin{equation*}
    F(z) \coloneqq f(z) \ \frac{\eta^5(z)}{\eta(5z)} = \sum_{m=0}^\infty c(m)q^m \in S_\frac{15}{2}(\Gamma_1(1440)).
\end{equation*}
Notice $\eta^5(z)/\eta(5z) \equiv 1 \pmod{5}$. So, $F(z) \equiv f(z) \pmod{5}$. Also, we have $\eta^{12}(80z)/q^{40} = 1 + \sum_{m > 0} a(80m)q^{80m}$. Thus, if we can show that
\begin{equation}
    b(80n) \equiv 0 \pmod{5} \text{ for all } n \geq 1,
\end{equation}
then $(\ref{congruence 1})$ holds. Since $80 \mid 1440$, we have 
\begin{equation*}
    F(z) \mid U_{80} = \sum_{n = 0}^\infty c(80n)q^n \in S_\frac{15}{2}(\Gamma_1(1440)). 
\end{equation*}
Recall that $b(80n) \equiv c(80n) \pmod{5}$. By Sturm's Criterion, we need to verify that $c(80n) \equiv 0 \pmod{5}$ for $0 \leq n \leq \frac{15}{24}[\SL_2(\mathbb{Z}) : \Gamma_0(1440)] + 1 = 2161$. We do this by using SAGE and computing the series $\sum_{m \geq 0} b(m)q^m$ up to $80\cdot2161 = 172,880$ terms and then focusing on the coefficients whose index is divisible by 80. This proves the first congruence. 

\bigskip

\noindent
The second congruence is proved in a nearly identical fashion. First, we define the eta-quotient
\begin{equation*}
    g(z) \coloneqq \frac{\eta(3z)}{\eta(2z)\eta(z)} \eta^6(80z) = \sum_{m \geq 0} \beta(m)q^m \in M_{\frac{5}{2}}^!(\Gamma_1(2880)).
\end{equation*}
Define
\begin{equation*}
    G(z) \coloneqq g(z) \frac{\eta^5(z)}{\eta(5z)} \in S_{\frac{9}{2}}(\Gamma_1(2880)).
\end{equation*}
We notice that $G(z) \equiv g(z) \pmod{5}$. Also, $\eta^6(80z)/q^{20} = q + \sum_{m > 0} \alpha(80m)q^{80m}$. Thus, if we can show that 
\begin{equation}
    \beta(80n) \equiv 0 \pmod{5} \text{ for all } n \geq 1,
\end{equation}
then $(\ref{congruence 2})$ holds. Since $80 \mid 2880$, then we have 
\begin{equation*}
    G(z) \mid U_{80} = \sum_{n = 0}^\infty \gamma(80n) q^n \in S_\frac{9}{2}(\Gamma_1(2880)).
\end{equation*}
Hence, $\beta(80n) \equiv \gamma(80n) \pmod{5}$. By Sturm's Criterion, we need to verify that $\gamma(80n) \equiv 0 \pmod{5}$ for $0 \leq n \leq 9/24 [\SL_2(\mathbb{Z} : \Gamma_0(2880)] + 1 = 2593$. We do this by using SAGE and computing the series $\sum_{m \geq 0} \beta(m)q^m$ up to $80\cdot2593 = 207440$ terms and then focus on the coefficients whose index is divisible by 80. This completes the proof.

\section{Example of Proposition 2} \label{exampleofProp2}
In this section, we give an explicit example of Proposition \ref{filtration} by following the steps outlined in the proof. Let $\ell = 7$ and 
\begin{equation*}
    f(z) = \frac{\eta(3z)}{\eta(2z)\eta(z)}.
\end{equation*}
The corresponding $F_7$ as defined in $\eqref{keys to theorem}$ is 
\begin{equation*}
    F_7(z) = \Delta^{49}(z)\left(\frac{\eta(z)\eta(2z)}{\eta(3z)}\right)^{48},
\end{equation*}
which is a holomorphic modular form for $\Gamma_1(6)$. We need to find a set of coset representatives for $\Gamma_1(6)$ in $\SL_2(\mathbb{Z})$. The following lemma will make this easier.
\begin{lem}\label{coset lemma} Let $\gamma, \gamma' \in \SL_2(\mathbb{Z})$.
\begin{equation*}
    \Gamma_1(N)\gamma = \Gamma_1(N)\gamma' \iff |(c,d)| = |(c',d')| = N \text{ and } (c,d) \neq (c',d') \text{ in } (\mathbb{Z}/N\mathbb{Z})^2,
\end{equation*}
where $\gamma = \begin{bmatrix} a & b \\ c & d\end{bmatrix}$ and $\gamma' = \begin{bmatrix} a' & b' \\ c' & d'\end{bmatrix}$.
\end{lem}
\begin{rem}
One needs to be slightly careful when applying the above lemma. For instance, when $N = 6$, $|(5,5)| = 6$ but since $c$ and $d$ in this case are not relatively prime, then one cannot construct a matrix $\begin{bmatrix} a & b \\ 5 & 5 \end{bmatrix} \in \SL_2(\mathbb{Z})$. However, $(5,5) \equiv (5,-1) \pmod{6}$ and this problem is resolved. The fact that this can always be resolved is Lemma 3.8.4 of \cite{diamond_shurman_2005}.
\end{rem}
The above lemma states that all we must find is the pairs $(c,d)$ in $(\mathbb{Z}/6\mathbb{Z})^2$ with order $6$. Furthermore, by consulting the formulas for computing $q$-expansions of eta-quotients in \cite{Ryan2019} and using the fact that we are raising the eta-quotient to the 48th power, we notice that the $q$-expansions are only dependent on these two matrix entries. There are 24 such elements of $(\mathbb{Z}/6\mathbb{Z})^2$, however, not all of these elements produce unique $q$-expansions. This is because $(c,d)$ and $(-c,-d)$ give rise to the same expansion and are not equivalent because $3 \nmid \gcd(c,d) = 1$. Thus, there are twelve expansions to compute. For example, 
\begin{equation*}
     F_7(z) \mid \begin{bmatrix}
    1 & 0 \\ 1 & 1
    \end{bmatrix} = \Delta^{51}(z) \frac{1}{2^{24}} q \left[\prod_{n=1}^\infty (1-(-q^{1/2})^{n})\right]^{48} 3^{24} \zeta_3 q^{-2/3} \left[\prod_{n=1}^\infty (1-(\zeta_3 q^{1/3})^n)\right]^{-48}.
\end{equation*}
By multiplying all of these expansions together,  we obtain a formula for $G$ as in the proof of  Proposition \ref{filtration},
\begin{equation*}
    G(z) = \frac{3^{432}}{2^{384}} q^4 \frac{\Delta^{1224}(z)\Delta^{16}(2z)}{\Delta^{12}(3z)} \prod_{n=1}^\infty \frac{(1-q^{n/2})^{384}(1-(-q^{1/2})^n)^{384}}{(1-q^{n/3})^{288} (1-(\zeta_3q^{1/3})^n)^{288}(1-(\zeta_3^2 q^{1/3})^n)^{288}},
\end{equation*}
which (as expected) is equal to $C \Delta^{1224}(z)$ where $C = \frac{3^{432}}{2^{384}}$. We will also show that $H$ (as in Proposition \ref{filtration} with $m = 1$) does not vanish mod $7$. To accomplish this, we use the fact $\theta(F_7 \mid M) \equiv (\theta F_7) \mid M \pmod{7}$
which comes from the proof of Lemma 4.2 in \cite{Sinick}. Using SAGE, we obtain
\begin{equation*}
    H \equiv 2q^{1252} + 5q^{1253} + q^{1254} + 2q^{1255} + 2q^{1256} + \cdots \pmod{7}
\end{equation*}
which is clearly not equivalent to 0.

\bibliographystyle{alpha}
\bibliography{bibliography}
\end{document}